\documentclass[leqno,english]{smfart}

\usepackage{bull} 
\usepackage{amssymb}
\usepackage{amsmath}
\usepackage{amsfonts}
\usepackage{amsthm}
\usepackage{verbatim}

\renewcommand{\phi}{\varphi}
\renewcommand{\epsilon}{\varepsilon}
\renewcommand{\hat}{\widehat}
\renewcommand{\tilde}{\widetilde}

\newcommand{\C}{\mathbb{C}}

\newcommand{\BB}{\mathcal{B}}
\newcommand{\BBh}{\hat{\BB}}

\newcommand{\FF}{\mathcal{F}}
\newcommand{\GG}{\mathcal{G}}

\newcommand{\HH}{H}
\newcommand{\HHH}{\mathcal{H}}

\newcommand{\NNT}{\tilde{N}}
\newcommand{\OO}{\mathcal{O}}
\newcommand{\OOh}{\hat{\OO}}

\newcommand{\Th}{\hat{T}}
\newcommand{\TT}{T}
\newcommand{\TTT}{\tilde{T}}

\newcommand{\WW}{\Xi}

\newcommand{\Xb}{\overline{X}}

\newcommand{\st}{\tilde{s}}
\newcommand{\ft}{\tilde{f}}
\newcommand{\fh}{\hat{f}}

\newcommand{\spr}{s^\prime}
\newcommand{\hhat}{\tilde{h}}
\newcommand{\htp}{\hhat^\prime}
\newcommand{\gp}{g^\prime}

\newcommand{\hhatinv}{(\tilde{h}^{-1})}

\newcommand{\varthetat}{\tilde{\vartheta}}
\newcommand{\thetat}{\tilde{\theta}}
\newcommand{\thetatp}{(\thetat^\prime)}
\newcommand{\etat}{\tilde{\eta}}
\newcommand{\etatp}{(\etat^\prime)}

\newcommand{\that}{\hat{\theta}}

\newcommand{\ehat}{\hat{\epsilon}}

\newcommand{\phat}{\hat{\phi}}

\newcommand{\nablat}{\hat{\nabla}}
\newcommand{\omegat}{\hat{\omega}}

\newcommand{\phicy}{\rho}
\newcommand{\hook}{\lrcorner\:}

\newcommand{\ddbar}{\partial\overline{\partial}}
\newcommand{\dbar}{\overline{\partial}}

\newcommand{\al}{\alpha}
\newcommand{\bbar}{\overline{\beta}}
\newcommand{\abar}{\overline{\alpha}}
\newcommand{\gbar}{\overline{\gamma}}
\newcommand{\lbar}{\overline{\lambda}}
\newcommand{\mbar}{\overline{\mu}}

\newcommand{\oneb}{\overline{1}}

\newcommand{\nb}{\overline{n}}

\newcommand{\Vol}{\textnormal{Vol}}
\newcommand{\const}{\textnormal{const}}
\renewcommand{\Re}{\textnormal{Re}}

\newcommand{\Ein}{\textnormal{Ein}}
\newcommand{\Einp}{\Ein^\prime}
\newcommand{\Ric}{\textnormal{Ric}}
\newcommand{\Riem}{\textnormal{Riem}}
\newcommand{\Scal}{\textnormal{Scal}}
\newcommand{\End}{\textnormal{End}}

\theoremstyle{plain} 
\newtheorem{thm}{Theorem}[section] 
\newtheorem{prop}[thm]{Proposition} 
\newtheorem{lem}[thm]{Lemma}

\theoremstyle{definition} 
\newtheorem{df}[thm]{Definition}  

\theoremstyle{remark} 
\newtheorem{rem}{Remark\!\!}

\theoremstyle{plain} 
\newtheorem{question}{Question\!\!}

\theoremstyle{plain} 
\newtheorem{theo}{Theorem~\ref{thm:ache}\!\!}

\theoremstyle{plain} 
\newtheorem{propo}{Proposition~\ref{prop:obs}\!\!}

\theoremstyle{plain} 
\newtheorem{mt}{Main Theorem\!\!}

\begin{document}   

\title[Approximately Einstein ACH metrics]{Approximately Einstein ACH metrics, volume renormalization, and an invariant for contact manifolds} 

\alttitle{M\'etriques presque d'Einstein ACH, renormalisation de volume, et un invariant pour les vari\'et\'es de contact}




\author{Neil Seshadri}   

\address{Graduate School of Mathematical Sciences, The University of Tokyo         
\newline\indent 3--8--1 Komaba, Meguro, Tokyo 153--8914, Japan} 
\email{seshadri@ms.u-tokyo.ac.jp}  

\begin{abstract} 
To any smooth compact manifold $M$ endowed with a contact structure $H$ and partially integrable almost CR structure $J$, we prove the existence and uniqueness, modulo high-order error terms and diffeomorphism action, of an approximately Einstein ACH (asymptotically complex hyperbolic) metric $g$ on $M\times (-1,0)$. 

We consider the asymptotic expansion, in powers of a special defining function, of the volume of $M\times (-1,0)$ with respect to $g$ and prove that the log term coefficient is independent of $J$ (and any choice of contact form $\theta$), i.e., is an invariant of the contact structure $H$. 

The approximately Einstein ACH metric $g$ is a generalisation of, and exhibits similar asymptotic boundary behaviour to, Fefferman's approximately Einstein complete K\"ahler metric $g_+$ on strictly pseudoconvex domains. The present work demonstrates that the CR-invariant log term coefficient in the asymptotic volume expansion of $g_+$ is in fact a contact invariant. We discuss some implications this may have for CR $Q$-curvature. 

The formal power series method of finding $g$ is obstructed at finite order. We show that part of this obstruction is given as a one-form on $H^*$. This is a new result peculiar to the partially integrable setting. 
\end{abstract}  

\begin{altabstract}
Pour toute vari\'et\'e lisse compacte $M$ munie d'une structure de contact $H$ et d'une structure presque CR partiellement int\'egrable $J$, nous d\'emontrons l'existence et l'unicit\'e, \`a des termes d'erreur de degr\'e sup\'erieur et action de diff\'eomorphisme pr\`es, d'une m\'etrique presque d'Einstein ACH (asymptotiquement complexe hyperbolique) $g$ sur $M\times (-1, 0)$.

Nous consid\'erons le d\'eveloppement asymptotique, en des puissances d'une fonction d\'efinissante sp\'eciale, du volume de $M\times (-1,0)$ par rapport \`a $g$. Nous d\'emontrons que le coefficient du terme logarithmique est ind\'ependant de $J$ (et du choix de la forme de contact $\theta$) ;  par cons\'equent, c'est un invariant de la structure de contact $H$.

La m\'etrique presque d'Einstein ACH $g$ est une g\'en\'eralisation de la m\'etrique presque d'Einstein k\"ahl\'erienne compl\`ete $g_{+}$ de Fefferman sur les domaines strictement pseudo-convexes. Elle a \'egalement un comportement asymptotique similaire au bord. Le pr\'esent travail d\'emontre que le coefficient du terme logarithmique CR-invariant dans le d\'eveloppement asymptotique du volume de $g_{+}$ est, en fait, un invariant de contact. Nous traitons \'egalement quelques implications possibles pour la $Q$-courbure CR.

La m\'ethode de trouver $g$ par le biais de s\'eries formelles comporte une obstruction d'ordre fini. Nous d\'emontrons que cette obstruction est partiellement donn\'ee par une $1$-forme sur $H^{*}$. Ceci est un r\'esultat nouveau particulier au contexte partiellement int\'egrable.
\end{altabstract}

\maketitle 

\newpage

\numberwithin{equation}{section}  

\section{Introduction}
In a previous paper~\cite{seshadri04}, inspired by Graham~\cite{graham}, we studied volume renormalization for Fefferman's approximately Einstein complete K\"ahler metric on a strictly pseudoconvex domain in a complex manifold. We considered the asymptotic expansion, in powers of a special boundary defining function, for the volume of this domain and showed that the coefficient $L$ of the log term in this expansion is an invariant of the boundary CR structure. In complex dimension two $L$ always vanishes. In higher dimensions, we showed in subsequent work~\cite{seshadri05} that $L$ is moreover invariant under (integrable) deformations of the CR structure. This led us to speculate that $L$ is in fact an invariant of the \emph{contact} structure on the boundary. One of the purposes of the present paper is to show that this is indeed the case.

To make sense of the last statement, in this paper we first generalise $L$ to be defined on an aribitrary smooth compact orientable contact manifold $(M,H)$. We do so by first endowing $(M,H)$ with a partially integrable almost CR structure $J$, a generalisation of an integrable CR structure. (Background material with definitions will follow in the next section.) We then define $L$ in this generalised setting as the log term coefficient in the volume expansion of an approximately Einstein ACH (asymptotically complex hyperbolic) metric $g$ on $X := M\times (-1,0)$. Our definition of ACH is contained is Definition~\ref{df:ach}; suffice for now to say that such a metric exhibits similar boundary asymptotics to those of Fefferman's approximately Einstein complete K\"ahler metric. The special defining function $\phi$ used for the expansion is from Lemma~\ref{lem:special} and corresponds to a choice of contact form $\theta$ for $(M,H)$. Then we have:

\begin{theo}
Define the \emph{Einstein tensor} by $\Ein := \Ric + 2(n+2)g$. Then there exists an ACH metric $g$ on $X$ that solves $\Ein = O(\phi^n)$ with $\Ein(W, Z) = O(\phi^{n+1})$ for $W\in H$ and $Z\in TM$. Moreover if $\gp$ is another such ACH metric then there exists a diffeomorphism $F$ of $\Xb$ that restricts to the identity on $M$ with $\gp = F^\ast g + \phi^n G$, where $G$ is O(1) and $G(W, Z) = O(\phi)$ for $W\in H$ and $Z\in TM$.
\end{theo}

There are two key initial steps in the proof of Theorem~\ref{thm:ache}, both reminiscent of similar steps in the integrable CR setting~\cite{seshadri04}. The first is to make a special choice of coframe for $T\Xb$ to allow $g$ to be written in a normal form---see \S\ref{sec:approx}. The second is to write the Levi-Civita connection and curvature of $g$ in terms of local data associated with (the extension to $\Xb$ of) a canonical connection adapted to $(M,H,J,\theta)$. We use a connection introduced by Tanno~\cite{tanno} and call it the TWT connection, since it reduces to the more familiar Tanaka--Webster connection when $J$ is integrable. Details about the TWT connection are in \S\ref{sec:pseudo}. 

The remainder of the proof of Theorem~\ref{thm:ache} uses methods from Graham--Hirachi~\cite{gh}. We solve the Einstein equation iteratively to determine $g$ up to a finite order and then use the contracted Bianchi identity to prove that all the components of $\Ein$ vanish to the correct order---see \S\ref{sec:einstein}.

The main result of this paper is the following:
\begin{mt}
The log term coefficient $L$ in the asymptotic volume expansion of $X$ with respect to an approximately Einstein ACH metric $g$ is an invariant of the contact structure $H$ on $M$.
\end{mt}

Since any contact manifold admits a contractible homotopy class of partially integrable almost CR structures, the proof of the Main Theorem follows from a deformation argument, similar to that in~\cite{gh} and~\cite{seshadri05}---see \S\ref{sec:vol}.

Now when the partially integrable almost CR structure $J$ is integrable, so that $(M,H,J)$ is a CR manifold, the approximately Einstein ACH metric from Theorem~\ref{thm:ache} does in fact coincide (modulo high order error terms and diffeomorphism action) as a Riemannian metric with Fefferman's approximately Einstein complete K\"ahler metric---see \S\ref{sec:integr}. The respective special defining functions also coincide, hence the log term coefficients $L$ coming from the two volume renormalization procedures (i.e., in this paper and~\cite{seshadri04}) agree.

When $M$ has dimension 3, $J$ is automatically integrable. In this dimension $L$ always vanishes (\cite{herzlich},~\cite{seshadri04}). Whether there exist nonzero $L$ in higher dimensions is an open question. Direct calculation using our volume renormalization techniques seems a computationally infeasible task. In the integrable CR setting the fact that $L$ is a constant multiple of the integral of CR $Q$-curvature makes settling the question of its (non)vanishing an even more pertinent task. The contact-invariance of $L$ proved in this paper could be a useful contribution to a solution to this problem. We speculate more on this matter and briefly discuss some other recently disovered contact invariants in the final \S\ref{sec:remarks}.

Some remarks are in order about the literature on ACH metrics. Our definition of ACH is closest to that of Guillarmou--S\'a Barreto~\cite{ga}, which in turn is based on the formalism of so-called $\Theta$-metrics from Epstein--Melrose--Mendoza~\cite{emm}. ACH-like metrics have also been studied by Roth~\cite{roth}, Biquard~\cite{biquard}, Biquard--Herzlich~\cite{bh} and Biquard--Rollin~\cite{br}; these authors also considered Einstein conditions, although with different purposes in mind to ours.

Finally let us make some comments about \emph{ACHE (ACH Einstein) metrics}, by which we mean ACH metrics satisfying $\Ein = O(\phi^m)$, for all $m$. With additional smoothness restraints, the existence of such metrics is in general obstructed by certain tensors. In \S\ref{sec:einstein} we define the obstruction tensors (for $T$ the Reeb field and $W_A$ in the contact direction)
$$
\BB := \phi^{-n}\Ein(T,T)|_M,
$$
$$
\OO_A := \phi^{-(n+1)}\Ein(T,W_A)|_M,
$$
and prove the following result.
\begin{propo}
\emph{(i)} The obstruction tensors $\BB$ (a scalar function) and $\OO_A$ are well-defined independently of the ambiguity in approximately Einstein ACH metric $g$.\\
\emph{(ii)} Under a change in contact form $\that = e^{2\Upsilon}\theta$, the obstruction tensors satisfy
\begin{equation*}
\BBh = e^{-2(n+2)\Upsilon}\BB
\end{equation*}
and
\begin{equation*}
\OOh_A = e^{-2(n+2)\Upsilon}(\OO_A -2i\phi^{-(n+1)}\Ein(\Upsilon^\alpha W_\alpha - \Upsilon^{\bbar}W_{\bbar}, W_A)|_M).
\end{equation*}
\emph{(iii)} If $(M, H, J)$ is such that $\BB$ vanishes then, under a change in contact form $\that = e^{2\Upsilon}\theta$, the obstruction $\OO_A$ satisfies
\begin{equation*}
\OOh_A = e^{-2(n+2)\Upsilon}\OO_A.
\end{equation*}
\end{propo} 
That there exists a secondary obstruction $\OO_A$ given as a one-form in $H^\ast$ is a novel feature of this partially integrable setting, since in the integrable case it is well-known (\cite{fef76},~\cite{lm},~\cite{graham87}) that the only obstruction to appear is a scalar function. Studying further the obstruction tensors and in particular their relation with the (almost) CR deformation complex should be interesting (cf.~\cite{gp} in the setting of conformal geometry). 

On the other hand, with no additional smoothness restraints, we expect the general question of existence of ACHE metrics to be settled by introducing log terms in the expansion for the metric. The formal theory for ACHE metrics closely resembles that of the ambient metric in conformal geometry, whose existence is proved by Fefferman--Graham~\cite{fg07}.

\textit{Notations and conventions.} Lowercase Greek indices run $1,\dots,n$. Uppercase Latin indices in $\{A, B,\dots,I\}$ run $1,\dots,n,\oneb,\dots,\nb$ while uppercase Latin indices in $\{J, K,\dots,Z\}$ run $\infty, 0, 1,\dots,n,\oneb,\dots,\nb$. The letter  $i$ will denote the quantity $\sqrt{-1}$. We observe the summation convention. Smooth for us means infinitely differentiable.

\textit{Acknowledgements.} This work was part of my PhD thesis at the University of Tokyo. I am most grateful to my supervisor Prof.~Kengo Hirachi for his expert guidance. I also thank Prof.~Robin Graham for helpful discussions and for his hospitality during my visit to the University of Washington in the summer of 2006. This project was commenced during my stay at the National Center for Theoretical Sciences at National Tsing-Hua University, Taiwan, in the summer of 2005; I would like to thank that institution for its hospitality and the organisers of the NCTS Mini-course \& Workshop ``Conformal Invariants---Geometric and Analytic Aspects'' for the invitation to attend. I am grateful for the financial support of a Japanese Government (MEXT) Scholarship for research students.

\section{Contact manifolds and ACH metrics in normal form}
\label{sec:approx}

\subsection{Contact manifolds}

Let $M$ be a smooth compact orientable manifold of dimension $2n+1$, endowed with a \emph{contact structure} $H$. That is, the smooth hyperplane distribution $H\subset TM$ is given as the kernel of a globally-defined nonvanishing one-form $\theta$ that satisfies the condition of maximal nonintegrability $\theta\wedge(d\theta)^n\neq 0$. If $\theta$ is such a \emph{contact form} then so is any smooth positive multiple of $\theta$.

The \emph{Reeb vector field} $T$ is characterised by the conditions $\theta(T) = 1$ and $T\hook d\theta = 0$.

Let $J\in\End(H)$ be a \emph{partially integrable almost CR structure}. That is, $J^2 = -1$ and the \emph{Levi metric} $2d\theta(\cdot,J\cdot)$ is Hermitian positive definite. (Such a $J$ is sometimes called a \emph{calibrated} or \emph{compatible} almost complex structure.) Note that the second condition is independent of choice of contact form, as is the conformal class of Levi metrics, denoted $[h]$, associated to $(H, J)$. Since $H$ is a symplectic vector bundle, there exists a contractible homotopy class of such $J$. 

\subsection{ACH metrics}

Consider the manifold-with-boundary 
$$
\Xb := M\times(-1, 0] \ni (x,\rho).
$$
Extend $H,J,\theta$ and $T$ to $\Xb$ by extending trivially in the $\rho$-direction. Throughout this paper $O(\rho^k)$ will denote quantities on $X$ that, when divided by $\rho^k$, extend at least continuously to $M = M\times\{0\}$. 

\begin{df}
\label{df:ach}
A smooth Riemannian metric $g\in S^2TX$ is said to be \emph{ACH (asymptotically complex hyperbolic)} if it satisfies:
\begin{enumerate}
\item $\rho g|_H = O(1)$ and on $M$, $\rho g|_H \in [h]$;
\item $\rho^2g = O(1)$ and $\rho^2 g|_M$ is a smooth multiple of $\theta^2|_M$;
\item $|(d\rho)/2\rho|_g^2 = O(1)$ and $|(d\rho)/2\rho|_g^2|_M = 1$;
\item $g(T, W) = O(1)$, for any $W\in H$;
\item $g^{-1}(\mu, \nu) = O(\rho)$, for any one-forms $\mu, \nu$ on $\Xb$;
\item $g^{-1}(d\rho, \mu) = O(\rho^2)$, for any one-form $\mu$ on $\Xb$.
\end{enumerate}
\end{df}

Note that our definition of an ACH metric depends on a choice of $J$ and $\theta$. Conditions (1), (2) and (3) in the definition are analogous to those used for asymptotically real hyperbolic metrics, see, e.g.,~\cite{fg01}. The motivation behind the remaining conditions will be apparent below.

\subsection{Normal form for ACH metrics}
\label{subsec:normal}
For the volume renormalization procedure it is convenient to work with $g$ in a normal form. As familiar from~\cite{graham}, ~\cite{seshadri04} and~\cite{ga}, this involves the choice of a special defining function for the boundary.

\begin{lem}
\label{lem:special}
There exists a unique defining function $\phi$ for $M$ in $\Xb$ such that
\begin{equation}
\label{eq:special}
\phi^2 g|_M = 4\theta^2|_M;\quad \left|\frac{d(\log(-\phi))}{2}\right|_g^2 = 1.
\end{equation}
\end{lem}
\begin{proof}
Write $\phi = e^{2f}\phicy$ for a function $f$ to be determined. The boundary value of $f$ is determined from the first condition in (\ref{eq:special}) above and assumption (2) in Definition~\ref{df:ach}. Next 
$$ 
1 = \left|\frac{d(\log(-\phi))}{2}\right|_g^2 = \frac{1}{4\rho^2}(|d\rho|_g^2 + 4\rho \langle d\rho, df \rangle_g + 4\rho^2|df|_g^2)
$$
is true if and only if
\begin{equation*}
\begin{split}
\frac{\partial f}{\partial p} &+\frac{1}{|d\rho|^2}\langle d\rho, dx \rangle_g\frac{\partial f}{\partial x} 
+ \frac{2\rho}{|d\rho|^2}\frac{\partial f}{\partial x}\frac{\partial f}{\partial \rho}\langle d\rho, dx \rangle_g\\
& +\rho\left(\frac{\partial f}{\partial\rho}\right)^2+ \frac{\rho}{|d\rho|^2}\left(\frac{\partial f}{\partial x}\right)^2|dx|^2 = \frac{\rho}{|d\rho|^2}\left(1 - \left|\frac{d\rho}{2\rho}\right|^2\right).
\end{split}
\end{equation*}
By assumptions (3), (5) and (6) in Definition~\ref{df:ach}, this is a noncharacteristic PDE with a unique solution $f$ near $M$.
\end{proof}

Define the vector field $\NNT$ as the dual of $d\phi/4\phi^2$, so that $\NNT := g(d\phi, \cdot)/4\phi^2$. Lemma~\ref{lem:special} shows that $d\phi(\NNT) = 1$, thus $\NNT$ is transverse to $M$. 

Define a diffeomorphism of $\Xb$ by mapping a point $(x,\rho)$ to the point in $\Xb$ obtained by following the unit-speed integral curve of $\NNT$ emanating from $x$ for time $\rho$. In the sequel, we shall omit from our notation pullbacks or pushforwards under this diffeomorphism.

Henceforth will shall work exclusively with the complexification of $H$, and abuse notation by denoting this as $H$ as well. The partially integrable almost CR structure $J$ extends to this complexified version, whence $H = H^{1,0}\oplus H^{0,1}$ splits into $i, -i$ eigenspaces. We shall work with the ACH metric $g$ naturally extended in the complexified bundle $H$. For local computations we let $\{W_A\}$ be a local frame for $H$.

\begin{lem} 
There exists a unique vector field $\TTT$ on $X$ near $M$ such that
$$
\TTT\perp_g\HH; \quad d\phi(\TTT) = 0;\quad \theta(\TTT) = 1.
$$
\end{lem}
\begin{proof}
Suppose $\TTT$ and $\TTT^\prime$ are two such vector fields. Then $\TTT - \TTT^\prime \in \ker\theta = \HH\oplus \NNT$, implying that $\TTT = \TTT^\prime$, and proving uniqueness.

For existence, set $\TTT = \TT - a^AW_A$ for functions $a^A$ to be determined by $g(\TTT, W_A) = 0$, i.e., $g(\TTT, W_A) = g(W_A, W_B)a^B$. The matrix on the right-hand side of this equation is nonsingular; this is by assumption (1) in Definition~\ref{df:ach}. Thus we may solve for the functions $a^B$, proving local, and hence global, existence of $\TTT$.
\end{proof} 

\begin{lem}
\label{lem:TM}
The vector field $\TTT$ extends to $M$ and $\TTT|_M = \TT|_M$.
\end{lem}
\begin{proof}
Locally write $\TTT = \TT - a^\alpha W_\alpha - b^{\bbar} W_{\bbar}$. Then
$$
0 = g(\TTT, W_{\bbar}) = g(\TT - a^\alpha W_\alpha - b^{\bbar} W_{\bbar}, W_{\bbar}) = O(1) - a^\alpha O(\phi^{-1}) - O(1),
$$
by assumptions (4) and (1) in Definition~\ref{df:ach}. We conclude that $a^\alpha = O(\phi)$. Similarly $b^{\bbar} = O(\phi)$.
\end{proof}

Lemma~\ref{lem:TM} shows that there are functions $\{\etat^A\}$, continuous up to $M$, such that
$$
\TTT = \TT - \phi\etat^A W_A.
$$
We shall henceforth work with the local frame $\{\NNT, \TTT, W_A\}$ for $T\Xb$ (near $M$). Let 
$\{d\phi, \theta, \thetat^A\}$ be the dual coframe, for some one-forms $\{\thetat^A\}$ that annihilate $\NNT, \TTT$. We specify them as follows: take an \emph{admissible coframe} for $H^{1,0}$, i.e., $(1,0)$-forms $\{\theta^\alpha\}$ satisfying $\theta^\alpha(W_\beta) = \delta_\beta^\alpha, \theta^\alpha(T) = 0$; set $\theta^{\abar} := \overline{\theta^\alpha}$; finally take
$$
\thetat^A := \theta^A + \phi\etat^A\theta.
$$
The reader is warned that the vector field $\TTT$ is \emph{not} purely real, so that in general
$$
\etat^{\abar}\neq\overline{\etat^\alpha}\textnormal{ and }\thetat^{\abar}\neq \overline{\thetat^\alpha}.
$$ 

Write 
\begin{equation}
\label{eq:levi}
d\theta = ih_{\alpha\bbar}\theta^\alpha\wedge\theta^{\bbar},
\end{equation}
for a positive definite Hermitian matrix $h_{\alpha\bbar}$. This Levi metric will be used to raise and lower indices.

With respect to our frame, the ACH metric $g$ has the normal form
\begin{equation}
\label{eq:diag}
g =  \left(\begin{array}{ccc} (2\phi)^{-2} & 0 & 0 \\ 0 & \phi^{-2}s & 0 \\ 0 & 0 & -\phi^{-1}\hhat_{AB} \end{array}\right), 
\end{equation}
for a function $s$ and matrix of functions $\hhat_{AB}$. By Lemma~\ref{lem:special}, $s|_M\equiv 4$. We moreover declare that
\begin{equation}
\label{eq:hhatM}
\hhat_{\alpha\bbar}|_M = h_{\alpha\bbar},
\end{equation}
with the other components of $\hhat$ vanishing on $M$; this is consistent with assumption (1) in Defintion~\ref{df:ach}. 

\section{Almost pseudohermitian geometry and the TWT connection}
\label{sec:pseudo}

Let $M$ be a smooth orientable manifold of dimension $2n+1$, endowed with a contact structure $H$, a partially integrable almost CR structure $J$, and a choice of contact form $\theta$. Then we call the quadruple $(M, H, J, \theta)$ an \emph{almost pseudohermitian manifold}. The terminology is inspired by Webster~\cite{webster}: when $J$ is integrable, i.e., satisfies $[H^{1,0},H^{1,0}]\subset H^{1,0}$, he called such objects pseudohermitian manifolds. 

Associated to any almost pseudohermitian manifold is a canonical connection introduced by Tanno~\cite{tanno}. It reduces to the Tanaka--Webster connection (\cite{tanaka},~\cite{webster}) in the integrable setting. We shall refer to it as the \emph{TWT connection}\footnote{This connection is sometimes called the ``generalised Tanaka--Webster connection'' in the literature. We prefer our terminology since there are other canonical connections that generalise Tanaka--Webster's; see \S\ref{sec:einstein} and~\cite{nicolaescu}.}. The reader may consult~\cite[Proposition 3.1]{tanno} for its axiomatic definition. Tanno chose to preserve Tanaka's torsion condition~\cite[Proposition 3.1(iii)]{tanno} from the integrable setting. The price paid is compatibility with $J$: unlike the Tanaka--Webster connection in the integrable case, the TWT connection does \emph{not} preserve the partially integrable almost CR structure. The extent to which $J$ is not preserved is measured by the \emph{Tanno tensor}, see Proposition~\ref{prop:twt} below, whose vanishing characterises integrable almost CR structures (\cite[Proposition 2.1]{tanno}).

Blair--Dragomir~\cite{bd} have further studied the TWT connection and its curvature and obtained many local formulae.

\begin{prop}[{Tanno~\cite[\S 6]{tanno}, Blair--Dragomir~\cite[\S 2.1]{bd}}]
\label{prop:twt}
\textnormal{(i)} Let the connection forms $\omega_{A}^{\phantom{A}B}$ of the TWT connection $\nabla$ be defined by
$$
\nabla W_{\alpha} = \omega_{\alpha}^{\phantom{\alpha}\beta}\otimes W_{\beta} + \omega_{\alpha}^{\phantom{\alpha}\bbar}\otimes W_{\bbar}; \quad\nabla T = 0. 
$$
Then the following structure equations are satisfied:
$$
d\theta^{\beta} = \theta^{\alpha}\wedge\omega_{\al}^{\phantom{\al}\beta} +  \theta^{\abar}\wedge\omega_{\abar} ^{\phantom{\abar}\beta} + A^{\beta}_{\phantom{\beta}\abar}\theta\wedge\theta ^{\overline{\al}}; \quad  \omega_{\al\bbar} + \omega_{\bbar\al} = dh_{\al\bbar}\:,
$$
where the torsion tensor $A$ satisfies
$$
A_{\alpha\beta} = A_{\beta\alpha}.
$$
\textnormal{(ii)} Let $Q(Y,X) = (\nabla_XJ)Y$, for $X, Y\in \C TM$ denote the \emph{Tanno tensor}. (Here $J$ is extended to all of $\C TM$ by declaring $JT = 0$). Then 
$$
Q^{\bbar}_{\phantom{\bbar}\alpha\gamma} := \theta^{\bbar}(Q(W_\alpha, W_\gamma)) = 2i\omega_{\al}^{\phantom{\al}\bbar}(W_\gamma),
$$
$$
Q^{\beta}_{\phantom{\beta}\abar\gbar} := \theta^{\beta}(Q(W_{\abar}, W_{\gbar})) = -2i\omega_{\abar}^{\phantom{\abar}\beta}(W_{\gbar}),
$$
and all other components of $Q$ vanish. Furthermore,
$$
\omega_{\al}^{\phantom{\al}\bbar}(W_{\gbar}) =  \omega_{\al}^{\phantom{\al}\bbar}(T)
= \omega_{\abar}^{\phantom{\abar}\beta}(W_\gamma) =  \omega_{\abar}^{\phantom{\abar}\beta}(T) = 0.
$$
\end{prop}

\begin{prop}[{Blair--Dragomir~\cite[Theorem 3]{bd}}]
\label{prop:twtcurv}
The curvature forms 
$$
\Omega_{\al}^{\phantom{\al}\beta} =  d\omega_{\al}^{\phantom{\al}\beta} - \omega_{\al}^{\phantom{\al}\gamma}\wedge \omega_{\gamma}^{\phantom{\gamma}\beta} -
\omega_{\al}^{\phantom{\al}\gbar}\wedge \omega_{\gbar}^{\phantom{\gbar}\beta}
$$
and
$$
\Omega_{\al}^{\phantom{\al}\bbar} =  d\omega_{\al}^{\phantom{\al}\bbar} -
\omega_{\al}^{\phantom{\al}\gamma}\wedge \omega_{\gamma}^{\phantom{\gamma}\bbar} - \omega_{\al}^{\phantom{\al}\gbar}\wedge \omega_{\gbar}^{\phantom{\gbar}\bbar} 
$$
of the TWT connection are given by
\begin{eqnarray*}  
\Omega_{\al}^{\phantom{\al}\beta} & = & R^{\phantom{\al}\beta}_{\al\phantom{\beta}\rho\overline{\gamma}}\theta^\rho\wedge\theta^ {\overline{\gamma}} + iA^\beta_{\phantom{\beta}\gbar}\theta_\alpha\wedge\theta^{\gbar} - iA_{\alpha\gamma}\theta^\gamma\wedge\theta^\beta\\ &&+ (A_{\alpha\gamma ,}^{\phantom{\alpha\gamma ,}\beta} + \frac{i}{2}Q_{\gamma\mu\alpha}A^{\mu\beta})\theta^\gamma\wedge\theta - (A^{\beta}_{\phantom{\beta}\overline{\gamma},\:\alpha} - \frac{i}{2}Q_{\gbar\mbar}^{\phantom{\gbar\mbar}\beta}
A^{\mbar}_{\phantom{\mbar}\alpha})\theta ^{\gbar} \wedge\theta\nonumber\\ 
&&-\frac{i}{4}Q_{\lambda\mu\alpha ,}^{\phantom{\lambda\mu\alpha ,}\beta}\theta^\lambda\wedge\theta^\mu - \frac{i}{4}Q_{\lbar\mbar\phantom{\beta} ,\:\alpha}^{\phantom{\lbar\mbar}\beta}\theta^{\lbar}\wedge\theta^{\mbar} \nonumber 
\end{eqnarray*} 
and
\begin{eqnarray*}  
\Omega_{\al}^{\phantom{\al}\bbar} & = &
(A_{\gamma\alpha\:,}^{\phantom{\gamma\alpha\:,}\bbar} -A_{\gamma\phantom{\bbar} ,\:\alpha}^{\phantom{\gamma}\bbar}) \theta^{\gamma}\wedge\theta 
+ \frac{i}{2}A_{\gbar\mbar}(Q^{\mbar\phantom{\alpha}\bbar}_{\phantom{\mbar}\alpha} - Q^{\mbar\bbar}_{\phantom{\mbar\bbar}\alpha})\theta^{\gbar}\wedge\theta\\
&&+\frac{i}{2}Q^{\bbar}_{\phantom{\bbar}\alpha\lambda ,\: \gamma}\theta^\lambda\wedge\theta^\gamma 
- \frac{i}{2}Q^{\bbar}_{\phantom{\bbar}\alpha\gamma ,\:\lbar}\theta^{\lbar}\wedge\theta^\gamma, 
\end{eqnarray*}
where a comma as a subscript indicates covariant differentiation with respect to the TWT connection.
\end{prop}
\begin{proof}
This is given in~\cite{bd}, with the exception of explicit formulae for the $\theta^\gamma\wedge\theta$ and $\theta^{\gbar}\wedge\theta$ terms. However a recipe for the calculation of these terms is given in the statement of~\cite[Theorem 3]{bd} and from this it is straightforward to obtain the formulae above.
\end{proof}

\section{Einstein equation and Bianchi identity}
\label{sec:einstein}

Let us now return to the ACH setup of \S~\ref{sec:approx}. Extend the TWT connection of $(M,H,J,\theta)$ to $\Xb$ by first trivially extending it in the $\NNT$-direction. Introduce the notation that $\infty$ as an index will denote the $\NNT$-direction and $0$ as an index will denote the $\TTT$-direction. Also in the interests of simplicity of notation we shall mainly work with indices $A, B,$ etc., avoiding $\alpha, \beta, \abar, \bbar,$ etc.~where possible. We need to keep in mind that certain components of some tensors will formally vanish. For instance, $h_{AB}$ vanishes unless $A = \alpha$ and $B = \bbar$ or vice versa. As another example, note that the formal quantity $A_B^{\phantom{B}B}$ (here $A$ is the torsion tensor from Proposition~\ref{prop:twt} and we have raised an index) also vanishes.

\begin{prop}
\label{prop:levicivita} 
Define a modified Kronecker symbol $\epsilon_A^B$ by
\begin{eqnarray*}
\epsilon_\alpha^B &:= &\delta_\alpha^B\\
\epsilon_{\abar}^B &:= &-\delta_{\abar}^B.
\end{eqnarray*}
Write the (extended) TWT connection forms as
$$ 
\omega_A^{\phantom{A}B} = \omega_{A\phantom{B}0}^{\phantom{A}B}\theta +
\omega_{A\phantom{B}C}^{\phantom{A}B}\theta^C.
$$
\nopagebreak
Then the Levi-Civita connection matrix $\psi_J^{\phantom{J}K}$ of $g$ with respect to the coframe $\{d\phi, \theta, \thetat^A\}$ satisfies
\begin{eqnarray*}
\psi_ \infty^{\phantom{\infty} \infty} & = &(-\phi^{-1})d\phi;\\
\psi_\infty^{\phantom{\infty}0} & =  &(-\phi^{-1} + \tfrac{1}{2}s^{-1}\NNT s)\theta\\ 
&& -\tfrac{1}{2}\phi s^{-1}\hhat_{AB}(\etat^B + \phi\NNT\etat^B)\thetat^A;\\ 
\psi_\infty^{\phantom{\infty}A} & =  &\tfrac{1}{2}(\etat^A + \phi\NNT\etat^A)\theta\\
&& + (\tfrac{1}{2}g^{AB}\NNT g_{BC})\thetat^{C};\\
\psi_0^{\phantom{0}\infty} & =  &(4\phi^{-1}s -2\NNT s)\theta\\ 
&& + 2\phi(\hhat_{AB}(\etat^B + \phi\NNT\etat^B))\thetat^A;\\  
\psi_ 0^{\phantom{ 0} 0} & = &(-\phi^{-1} + \tfrac{1}{2}s^{-1}\NNT s)d\phi\\
&& + (\tfrac{1}{2}s^{-1}\TTT s)\theta\\
&& + (\tfrac{1}{2}s^{-1}W_A s)\thetat^A;\\
\psi_0^{\phantom{0}A} & =  &-\tfrac{1}{2}(\etat^A + \phi\NNT\etat^A)d\phi\\
&& +s\phi^{-1}\hhatinv^{AB}(\tfrac{1}{2}s^{-1}W_Bs -i\epsilon_B^{C}\phi\etat_C)\theta\\
&& +\tfrac{1}{2}\hhatinv^{AB}\big(\TTT\hhat_{BC} + \hhat_{BD}(-\omega_{C\phantom{D}0}^{\phantom{C}D} 
+ A^D_{\phantom{D}C}- \phi W_{C}\etat^D\\
&&\quad\quad\quad + \phi\etat^E\omega_{C\phantom{D}E}^{\phantom{C}D} - \phi\etat^E\omega_{E\phantom{D}C}^{\phantom{E}D} + i\phi^2\epsilon^{F}_{C}\etat^D\etat_F)\\
&&\quad\quad + \hhat_{CD}(-\omega_{B\phantom{D}0}^{\phantom{B}D} 
+ A^D_{\phantom{D}B}- \phi W_{B}\etat^D\\
&&\quad\quad\quad + \phi\etat^E\omega_{B\phantom{D}E}^{\phantom{B}D} - \phi\etat^E\omega_{E\phantom{D}B}^{\phantom{E}D} + i\phi^2\epsilon^{F}_{B}\etat^D\etat_F)\\
&&\quad\quad - is\phi^{-1}\epsilon^D_C h_{BD}\big)\thetat^C;\\
\psi_A^{\phantom{A}\infty} & =  &2\phi\hhat_{AB}(\etat^B + \phi\NNT \etat^B)\theta\\
&& - 2\phi^2\NNT g_{AC}\thetat^{C};\\
\psi_A^{\phantom{A}0} & =  &-\tfrac{1}{2}\phi s^{-1}\hhat_{AB}(\etat^B + \phi\NNT\etat^B)d\phi\\
&& +(\tfrac{1}{2}s^{-1}W_As - i\phi\epsilon_A^B\etat_B)\theta\\
&& +\phi s^{-1}\hhat_{AB}\psi_{0\phantom{B}C}^{\phantom{0}B}\thetat^C;\\
\psi_A^{\phantom{A}B} & =  &\tfrac{1}{2}\phi \hhatinv^{BC}\NNT (\phi^{-1}\hhat_{AC})d\phi\\
&& + (\psi_{0\phantom{B}A}^{\phantom{0}B} +\omega_{A\phantom{B}0}^{\phantom{A}B} 
- A^B_{\phantom{B}A}+ \phi W_{A}\etat^B - \phi\etat^C\omega_{A\phantom{B}C}^{\phantom{A}B} + \phi\etat^C\omega_{C\phantom{B}A}^{\phantom{C}B}\\
&&\quad\quad - i\phi^2\epsilon^{C}_{A}\etat^B\etat_C)\theta\\
&& \tfrac{1}{2}\hhatinv^{BD}\big(W_C\hhat_{AD} - W_D\hhat_{AC} - W_A\hhat_{CD}
- \hhat_{AE}(\omega_{D\phantom{E}C}^{\phantom{D}E} - \omega_{C\phantom{E}D}^{\phantom{C}E})\\ 
&&\quad - \hhat_{CE}(\omega_{D\phantom{E}A}^{\phantom{D}E} - \omega_{A\phantom{E}D}^{\phantom{A}E}) 
- \hhat_{DE}(\omega_{C\phantom{E}A}^{\phantom{C}E} - \omega_{A\phantom{E}C}^{\phantom{A}E}) \big)\thetat^C.
\end{eqnarray*} 
\end{prop}
\begin{proof}
Set $\thetat^\infty := d\phi$ and $\thetat^0 := \theta$ and then simultaneously solve the structure equations
$$
d\thetat^J = \thetat^K\wedge \psi_K^{\phantom{K}J}
$$
and
$$
\psi_{JK} + \psi_{KJ} = dg_{JK},
$$
using Proposition~\ref{prop:twt} and (\ref{eq:levi}). The calculation is straightforward but long so we omit it.
\end{proof}

We introduce notation to keep track of derivatives in the $\NNT$-direction. Let $\WW^{(k)}$ generically denote a known $O(1)$ tensor on $X$ that contains at most $k$ iterated derivatives in the $\NNT$-direction. For example, $\WW_{AB}^{(2)}$ could denote the tensor $\NNT^2\hhat_{AB}$. In general such a tensor will involve tensorial invariants comprised of (Levi metric contractions of) TWT curvature, torsion and Tanno tensors ($R$, $A$ and $Q$ in Proposition~\ref{prop:twtcurv}) and their TWT covariant derivatives.

Now we present a result that may be used to simplify many tensorial calculations. In particular we use it when computing the Ricci tensor of $g$, modulo certain $\WW^{(k)}$ tensors, in Proposition~\ref{prop:ric} below. Moreover the result would be a useful computational tool if one wanted to calculate explicity the $\WW^{(k)}$ tensors.

\begin{lem}
\label{lem:frame}
Near any point $x\in X$, there exists a frame $\{W_A\} = \{W_\alpha, W_{\abar}\}$ for $H$ with respect to which the part of the TWT connection matrix given by $\omega_\beta^{\phantom{\beta}\alpha}$ vanishes at $x$. 
\end{lem}
\begin{proof}
Start by choosing any frame $\{W_A^x\}$ for $H$ at $x$. If we extend this frame to a neighbourhood of $x$ by TWT-parallel translation along the geodesics of the TWT connection $\nabla$, then a standard result from the theory of linear connections ensures that the resulting frame is smooth near $x$ and the (full) TWT connection matrix with respect to this frame vanishes at $x$. The problem is that, since the TWT connection does not preserve the almost CR structure $J$, the resulting frame is not of the form $\{W_\alpha, W_{\abar}\}$ (types are not preserved), meaning we cannot subsequently use the local almost pseudohermitian formulae from \S\ref{sec:pseudo}. We get around this problem by instead parallel translating the initial frame along geodesics with respect to a new connection $\nablat$, defined by
$$
\nablat_XY = \frac{1}{2}\left(\nabla_XY - J\nabla_XJY\right).
$$
One checks that $\nablat J = 0$ so that frame types are preserved under $\nablat$-parallel translation. Moreover, writing $\nablat W_{\alpha} = \omegat_{\alpha}^{\phantom{\alpha}\beta}\otimes W_{\beta}$, one checks that $\omegat_{\alpha}^{\phantom{\alpha}\beta} = \omega_{\alpha}^{\phantom{\alpha}\beta}$. This proves the lemma. 
\end{proof}

\begin{rem}
The connection $\nablat$ clearly reduces to $\nabla$ when $J$ is integrable, i.e., when $\nabla$ is the usual Tanaka--Webster connection preserving $J$. Thus $\nablat$ provides an alternative extension of the Tanaka--Webster connection to almost pseudohermitian manifolds. Of course Tanaka's torsion condition will no longer hold in general for $\nablat$. If one considers an analogy between almost pseudohermitian and almost Hermitian manifolds, then the TWT connection could be thought of as corresponding to the Levi-Civita connection and $\nablat$ to the first canonical connection of Lichnerowicz (see, e.g.,~\cite{ffs}, \cite{gauduchon}).
\end{rem}

If we work in a special frame given by Lemma~\ref{lem:frame}, then we can delete any TWT connection forms of type $\omega_\beta^{\phantom{\beta}\alpha}$ and replace ordinary derivatives with TWT covariant derivatives, plus Tanno tensor correction terms. This is because the remaining part $\omega_{\bbar}^{\phantom{\bbar}\alpha}$ of the TWT connection matrix may be identified with the Tanno tensor $Q$---see Proposition~\ref{prop:twt}(ii). An upshot of this is that, for example, we can delete from calculations terms of the form $\omega_{A\phantom{B}B}^{\phantom{A}B}$, as the reader may verify by recalling from Proposition~\ref{prop:twt}(ii) the form of the Tanno tensor.

\begin{prop}
\label{prop:ric}
The components of the Ricci tensor of $g$ satisfy
\begin{eqnarray*}
\Ric_{\infty\infty} &=& -\frac{1}{2}\left((n+2)\phi^{-2} - (s\phi)^{-1}\NNT s + s^{-1}\NNT^2s + 
\hhatinv^{AB}\NNT^2\hhat_{AB}\right)\\
&&\quad + \WW^{(1)};\\
\Ric_{\infty 0} &=& \WW^{(1)};\\
\Ric_{\infty A} &=& \WW_A^{(1)};\\
\Ric_{00} &=& -2(n+2)\phi^{-2}s + 2(n+1)\phi^{-1}\NNT s -2\NNT^2 s +\WW^{(1)};\\
\Ric_{0A} &=& -4(n+1)\hhat_{AB}\etat^B - 2(n-2)\phi\hhat_{AB}\NNT\etat^B + 2\phi^2\hhat_{AB}\NNT^2\etat^B\\
&&\quad + \WW_A^{(0)};\\
\Ric_{AB} &=& 2(n+1)\phi^{-1}\hhat_{AB} + s(2\phi)^{-1}\hhatinv^{CD}h_{AC}h_{DB}- 2n\NNT\hhat_{AB}\\
&&\quad - \hhatinv^{CD}(\NNT\hhat_{CD})\hhat_{AB} - s^{-1}(\NNT s)\hhat_{AB} + 2\phi\NNT^2\hhat_{AB}\\
&&\quad\quad + \WW_{AB}^{(0)} + \phi\WW_{AB}^{(1)}.
\end{eqnarray*}
Furthermore the terms containing $\etat$ in $\WW_A^{(0)}, \WW_{AB}^{(0)}$, and $\NNT\eta$ in $\WW^{(1)},\WW_A^{(1)}, \WW_{AB}^{(1)}$, are $O(\phi)$.
\end{prop}
\begin{proof}
The Ricci tensor of $g$ satisfies $\Ric_{JK} = g^{LM}\Riem_{JLMK}$, where $\Riem$ denotes the Riemannian curvature tensor. Using the normal form (\ref{eq:diag}) of $g$ we have
$$
\Ric_{JK} = \phi^2s^{-1}\Riem_{J00K} - \phi\hhatinv^{AB}\Riem_{JABK} + 4\phi^2\Riem_{J\infty\infty K}.
$$
Then with 
\begin{equation}
\label{eq:Psi}
\Psi_J^{\phantom{J}K} := d\psi_J^{\phantom{J}K} - \psi_J^{\phantom{J}L}\wedge\psi_L^{\phantom{L}K}
\end{equation}
denoting the Levi-Civita curvature matrix of $g$,
\begin{equation}
\label{eq:Ric}
\begin{split}
\Ric_{JK} &= \phi^2s^{-1}g_{LK}\Psi_0^{\phantom{0}L}(W_J, \TTT) - \phi\hhatinv^{AB}g_{LK}\Psi_B^{\phantom{B}L}(W_J, W_A)\\ 
&\quad\quad + 4\phi^2g_{LK}\Psi_\infty^{\phantom{\infty}L}(W_J, \NNT) .
\end{split}
\end{equation}
We then just compute using Proposition~\ref{prop:levicivita}. By Lemma~\ref{lem:frame} the calculations may be simplified by working in a frame where the connection forms $\omega_\beta^{\phantom{\beta}\alpha}$ vanish at a point. Again we omit the details.
\end{proof}

\begin{thm}
\label{thm:ache}
Define the \emph{Einstein tensor} by $\Ein := \Ric + 2(n+2)g$. Then there exists an ACH metric $g$ on $X$ that solves $\Ein = O(\phi^n)$ with $\Ein(W, Z) = O(\phi^{n+1})$ for $W\in H$ and $Z\in TM$. Moreover if $\gp$ is another such ACH metric then there exists a diffeomorphism $F$ of $\Xb$ that restricts to the identity on $M$ with $\gp = F^\ast g + \phi^n G$, where $G$ is O(1) and $G(W, Z) = O(\phi)$ for $W\in H$ and $Z\in TM$.
\end{thm}
\begin{proof}
Our approach is analogous to that of Graham--Hirachi~\cite[Theorem 2.1]{gh}. We prove that the system of equations $\Ein_{00} = O(\phi^n), \Ein_{0A} = O(\phi^{n+1})$ and $\Ein_{AB} = O(\phi^{n+1})$ is necessary and sufficient to uniquely determine, modulo the high-order error terms, the metric $g$. We then apply the contracted Bianchi identity to show that for this metric all the remaining components of $\Ein$ vanish to the correct order.

To begin, if the Einstein tensor vanishes then from Proposition~\ref{prop:ric} it follows that
\begin{eqnarray}
\label{eq:00} \phi \Ein_{00} &=& 2(n+1)\NNT s -2\phi\NNT^2 s +\phi\WW^{(1)} = 0;\\
\label{eq:0A} \Ein_{0A} &=& -4(n+1)\hhat_{AB}\etat^B - 2(n-2)\phi\hhat_{AB}\NNT\etat^B\\
&&\quad + 2\phi^2\hhat_{AB}\NNT^2\etat^B + \WW_A^{(0)} = 0;\nonumber\\
\label{eq:AB} \quad\quad\phi \Ein_{AB} &=& 2\hhat_{AB} - \frac{1}{2}s\hhatinv^{CD}h_{AC}h_{DB} + \phi\Big(2n\NNT\hhat_{AB}\\
&&\ +\hhatinv^{CD}(\NNT\hhat_{CD})\hhat_{AB} + s^{-1}(\NNT s)\hhat_{AB} - 2\phi\NNT^2\hhat_{AB}\Big)\nonumber\\
&&\quad+ \phi\WW_{AB}^{(0)} + \phi^2\WW_{AB}^{(1)} = 0.\nonumber
\end{eqnarray}
Observe that this system is tensorial hence the metric $g$ coming from its solution is well-defined independently of any frame choices. 

Applying $\NNT$ to (\ref{eq:00}) $(k-1)$-times and evaluating on $M$ gives
\begin{equation}
\label{eq:sk}
2\{n+2-k\}\NNT^ks\Big|_M = \textnormal{ (terms involving $\NNT^l\hhat_{AB}, \NNT^ls, \NNT^{l-1}\etat^B$ ($l<k$) on $M$)}.
\end{equation}
We have used here the last sentence in Proposition~\ref{prop:ric}.

Next applying $\NNT$ to (\ref{eq:AB}) $k$ times and evaluating on $M$ gives
\begin{equation}
\label{eq:hk}
\begin{split}
&2\{k^2 - (n+1)k - 2\}\NNT^k\hhat_{AB}\Big|_M = \left(\frac{k}{4} - \frac{1}{2}\right)(\NNT^ks)h_{AB} + kh^{\delta\gbar}(\NNT^k\hhat_{\delta\gbar})h_{AB}\Big|_M\\
&\quad + \textnormal{ (terms involving $\NNT^l\hhat_{AB}, \NNT^ls, \NNT^{l-1}\etat^B$ ($l<k$) on $M$)}.
\end{split}
\end{equation}
Taking the trace of this equation using the Levi metric $h$ gives
\begin{equation}
\label{eq:trhk}
\begin{split}
&\{k^2 - (2n+1)k - 2\}h^{\alpha\bbar}\NNT^k\hhat_{\alpha\bbar}\Big|_M = \\
&\  \frac{n}{4}\left(k - 2\right)\NNT^ks\Big|_M + \textnormal{ (terms involving $\NNT^l\hhat_{AB}, \NNT^ls, \NNT^{l-1}\etat^B$ ($l<k$) on $M$)}.
\end{split}
\end{equation}

Finally applying $\NNT$ to (\ref{eq:0A}) $k$-times and evaluating on $M$ gives
\begin{equation}
\label{eq:ek}
\begin{split}
&\{(k-(n+1))(k+2)\}h_{\alpha\bbar}\NNT^k\etat^{\bbar}\Big|_M = \\
&\quad\quad \textnormal{ (terms involving $\NNT^l\hhat_{AB}, \NNT^ls, \NNT^l\etat^B$ ($l<k$) on $M$)}.
\end{split}
\end{equation}

Inspection shows that we can now use (\ref{eq:sk}), (\ref{eq:hk}), (\ref{eq:trhk}) and (\ref{eq:ek}) to iteratively determine the desired Taylor coefficients $\NNT^ks\Big|_M, \NNT^k\hhat_{AB}\Big|_M$, up to some finite order. Obstructions occur when $k$ is such that the contents of one of the sets of braces in (\ref{eq:sk}), (\ref{eq:hk}), (\ref{eq:trhk}) or (\ref{eq:ek}) vanish. 

The contents of the braces in (\ref{eq:hk}) never vanish because they do so if and only if
$$
k = \frac{n+1 \pm \sqrt{(n+1)^2 + 8}}{2},
$$
and $(n+1)^2 + 8$ can never be a perfect square. Similarly the contents of the braces in (\ref{eq:trhk}) never vanish. 

Obstructions do occur when $k = n+1$ in (\ref{eq:ek}) or $k = n+2$ in (\ref{eq:sk}). This means that we can determine the power series for $\etat^A$ only up to the Taylor coefficient of $\phi^{n}$, and that of $s$ and $h^{\alpha\bbar}\hhat_{\alpha\bbar}$ (and hence $\hhat_{\alpha\bbar}$ itself) only up to the Taylor coefficient of $\phi^{n+1}$. In fact $\hhat_{\alpha\beta}$ (and $\hhat_{\abar\bbar}$) and the trace-free part of $\hhat_{\alpha\bbar}$ are determined to one order higher. Indeed, setting $k = n+2$ in (\ref{eq:hk}) and using (\ref{eq:trhk}) gives
\begin{equation}
\label{eq:htf}
\begin{split}
&\NNT^{n+2}\hhat_{AB} - (2n)^{-1}h^{\delta\gbar}(\NNT^{n+2}\hhat_{\delta\gbar})h_{AB}\Big|_M = \\
&\quad + \textnormal{ (terms involving $\NNT^l\hhat_{AB}, \NNT^ls, \NNT^{l-1}\etat^B$ ($l<n+2$) on $M$)},
\end{split}
\end{equation}
whence our claim.

At this stage we have shown that $\Ein_{00} = O(\phi^n), \Ein_{0A} = O(\phi^{n+1})$ and $\Ein_{AB} = O(\phi^{n+1})$. In order to show that $\Ein_{\infty J} = O(\phi^n)$, we use the contracted Bianchi identity
$$
g^{JK}D_L\Ric_{JK} = 2g^{JK}D_J\Ric_{KL},
$$
where here $D$ denotes the Levi-Civita connection of $g$. Since $D$ is compatible with $g$ we have
\begin{equation}
\label{eq:bianchi}
g^{JK}D_L\Ein_{JK} = 2g^{JK}D_J\Ein_{KL}.
\end{equation}

Introduce the notation that indices $\Pi$ and $\Sigma$ will run $0,1,\dots,n,\oneb,\dots,\nb$. Set $L= \infty$ in (\ref{eq:bianchi}) and convert the expression to a local one involving Levi-Civita connection forms:
\begin{equation}
\label{eq:e00}
\begin{split}
&4\phi\NNT \Ein_{\infty\infty} + 2\hhatinv^{AB}\Psi_{B\phantom{\infty}A}^{\phantom{B}\infty}\Ein_{\infty\infty} - 2\phi s^{-1}\Psi_{0\phantom{\infty}0}^{\phantom{0}\infty}\Ein_{\infty\infty} =\\
&\quad\phi s^{-1}(\NNT \Ein_{00} - 2\Psi_{0\phantom{\Pi}\infty}^{\phantom{0}\Pi}\Ein_{\Pi 0})\\
&\quad\quad- \hhatinv^{AB}(\NNT \Ein_{AB} - \Psi_{A\phantom{J}\infty}^{\phantom{A}J}\Ein_{JB} - \Psi_{B\phantom{J}\infty}^{\phantom{B}J}\Ein_{AJ})\\
&\quad\quad\quad + 2\hhatinv^{AB}(W_A\Ein_{B\infty}  - \Psi_{B\phantom{\Pi}A}^{\phantom{B}\Pi}\Ein_{\Pi\infty})\\
&\quad\quad\quad\quad -2\phi s^{-1}(\TTT\Ein_{0\infty} - \Psi_{0\phantom{\Pi}0}^{\phantom{0}\Pi}\Ein_{\Pi\infty}).
\end{split}
\end{equation}
Next taking $L=\Pi$ in (\ref{eq:bianchi}) leads to
\begin{equation}
\label{eq:eic}
\begin{split}
&8\phi\NNT \Ein_{\infty \Pi} + 8\Ein_{\infty\Pi} + 8\phi \Psi_{\infty\phantom{\Sigma}\Pi}^{\phantom{\infty}\Sigma}\Ein_{\Sigma\infty}\\
&-2\phi s^{-1}(\Psi_{0\phantom{\infty}0}^{\phantom{0}\infty}\Ein_{\infty\Pi} + \Psi_{0\phantom{\infty}\Pi}^{\phantom{0}\infty}\Ein_{\infty 0})\\
&+ \hhatinv^{AB}(2\Psi_{B\phantom{\infty}A}^{\phantom{B}\infty}\Ein_{\infty \Pi} - \Psi_{A\phantom{\infty}\Pi}^{\phantom{A}\infty}\Ein_{\infty B} -
\Psi_{B\phantom{\infty}\Pi}^{\phantom{B}\infty}\Ein_{\infty A}) =\\
&\quad \phi s^{-1}(W_\Pi\Ein_{00} - 2\Psi_{0\phantom{\Sigma}\Pi}^{\phantom{0}\Sigma}\Ein_{\Sigma 0})\\
&\quad\quad -2\phi s^{-1}(\TTT\Ein_{0\Pi} - \Psi_{0\phantom{\Sigma}0}^{\phantom{0}\Sigma}\Ein_{\Sigma\Pi})\\
&\quad\quad\quad +2\hhatinv^{AB}(W_A\Ein_{B\Pi} - \Psi_{B\phantom{\Sigma}A}^{\phantom{B}\Sigma}\Ein_{\Sigma\Pi})\\
&\quad\quad\quad\quad -\hhatinv^{AB}(W_\Pi\Ein_{AB} - \Psi_{A\phantom{\Sigma}\Pi}^{\phantom{A}\Sigma}\Ein_{\Sigma B} - \Psi_{B\phantom{\Sigma}\Pi}^{\phantom{B}\Sigma}\Ein_{A\Sigma})\\
&\quad\quad\quad\quad\quad + 4\phi W_\Pi\Ein_{\infty\infty}.
\end{split}
\end{equation}

We now argue by induction, supposing that $\Ein_{\infty\infty}$ and $\Ein_{\infty \Pi}$ vanish to orders $O(\phi^{t-1})$ and $O(\phi^t)$ respectively, for some integer $t$. Proposition~\ref{prop:ric} tells us that this is true when $t=0$. Suppose $\Ein_{\infty\infty} = \lambda\phi^{t-1}$ for a function $\lambda$. Inserting this formula into (\ref{eq:e00}), a short calculation using Proposition~\ref{prop:levicivita} and the known orders of vanishing of the other components of $\Ein$ gives that
$$
(t-1-2(n+1))\lambda = O(\phi), \textnormal{ for } t-1 < n.
$$
So provided $t-1 < n$ our induction argument shows that in fact $\lambda = O(\phi)$, i.e., $\Ein_{\infty\infty}$ is $O(\phi^{t})$.

On the other hand, suppose $\Ein_{\infty\Pi} = \mu_{\Pi}\phi^t$. Insert this formula into (\ref{eq:eic}) and a short calculation implies that
$$
(t-(n+2))\mu_{0} = O(\phi), \textnormal{ for } t < n,
$$
and
$$
(t-n)\mu_{A} = O(\phi), \textnormal{ for } t < n.
$$
So provided $t<n$ we have that $\mu_{\Pi} = O(\phi)$, i.e., $\Ein_{\infty\Pi}$ is $O(\phi^{t+1})$. 
We conclude that $\Ein_{\infty\infty}$ and $\Ein_{\infty \Pi}$ all vanish to orders $O(\phi^n)$, and this finishes the proof of the first assertion of the theorem.

Our argument thus far has also shown uniqueness of $g$, modulo high order error terms, provided $g$ is of the normal form (\ref{eq:diag}). But any ACH metric can be put into this normal form by using a special defining function (see \S~\ref{subsec:normal}). And changing defining function is just a diffeomorphism of $\Xb$ that restricts to the identity on $M$. This completes the proof of the theorem.
\end{proof}

Henceforth we refer to any metric $g$ satisfying the hypotheses of Theorem~\ref{thm:ache} as an \emph{approximately Einstein ACH metric}.

The proof of Theorem~\ref{thm:ache} shows that in general there are obstructions to the existence of an ACH metric $g$ on $X$ that solves $\Ein = O(\phi^m)$, for all $m$, with $\phi^2g$ and $\phi g|_H$ extending smoothly to $M$. Indeed if we define the \emph{obstruction tensors}
$$
\BB := \phi^{-n}\Ein(T,T)|_M
$$
and
$$
\OO_A := \phi^{-(n+1)}\Ein(T,W_A)|_M,
$$
then (\ref{eq:sk}) and (\ref{eq:ek}) imply that the vanishing of the $\BB$ and $\OO_A$ is a necessary condition for the existence of a metric $g$ just described. Observe that $\BB$ is a scalar function whereas $\OO_A$ is a section of the bundle $\C T^*M/\langle\theta\rangle$, with $\langle\theta\rangle$ denoting the ideal generated by $\theta$. (The bundle $\C T^*M/\langle\theta\rangle$ is defined independently of the choice of $\theta$ but after choosing a contact form may be identified with $H^*$.) Equations (\ref{eq:sk}) and (\ref{eq:ek}) also show that $\BB$ and $\OO_A$ are given as (Levi metric contractions of) TWT curvature, torsion and Tanno tensors and their TWT covariant derivatives.

The obstructions $\BB$ and $\OO_A$, as they are defined, a priori depend on choices of approximately Einstein ACH metric $g$ and contact form $\theta$. To make sense then of the necessary condition stated in the last paragraph, we need to confirm that $\BB$ and $\OO_A$ are in fact independent of the freedom of choice of approximately Einstein ACH metric, and also that their vanishing is independent of choice of contact form.

\begin{prop}
\label{prop:obs}
\emph{(i)} The obstruction tensors $\BB$ and $\OO_A$ are well-defined independently of the ambiguity in approximately Einstein ACH metric $g$.\\
\emph{(ii)} Under a change in contact form $\that = e^{2\Upsilon}\theta$, the obstruction tensors satisfy
\begin{equation}
\label{eq:BBh}
\BBh = e^{-2(n+2)\Upsilon}\BB
\end{equation}
and
\begin{equation}
\label{eq:OOAh}
\OOh_A = e^{-2(n+2)\Upsilon}(\OO_A -2i\phi^{-(n+1)}\Ein(\Upsilon^\alpha W_\alpha - \Upsilon^{\bbar}W_{\bbar}, W_A)|_M).
\end{equation}
\emph{(iii)} If $(M, H, J)$ is such that $\BB$ vanishes then, under a change in contact form $\that = e^{2\Upsilon}\theta$, the obstruction $\OO_A$ satisfies
\begin{equation}
\label{eq:OOA}
\OOh_A = e^{-2(n+2)\Upsilon}\OO_A.
\end{equation} 
\end{prop}
\begin{proof}
To prove (i) we fix a contact form and then consider another approximately Einstein ACH metric $\gp$ with Einstein tensor $\Einp$ and obstructions $\BB^\prime$ and $\OO_A^\prime$. Writing $\gp$ in normal form as
$$
\gp =  \left(\begin{array}{ccc} (2\phi)^{-2} & 0 & 0 \\ 0 & \phi^{-2}\spr & 0 \\ 0 & 0 & -\phi^{-1}\htp_{AB} \end{array}\right), 
$$
Theorem~\ref{thm:ache} and its proof imply that
\begin{equation}
\label{eq:sp}
\spr = s + \phi^{n+2}\kappa
\end{equation}
and
\begin{equation}
\label{eq:htp}
\htp_{AB} = \hhat_{AB} + \phi^{n+2}\lambda_{AB},
\end{equation}
where $\kappa$ and $\lambda_{AB}$ are $O(1)$. Also if we write the modified coframe (recall \S\ref{sec:approx}) for $\gp$ as 
$$
\thetatp^A = \theta^A + \phi\etatp^A\theta,
$$
we have that for some $\mu^A$ of order $O(1)$,
\begin{equation}
\label{eq:ep}
\etatp^A = \etat^A + \phi^{n+1}\mu^A.
\end{equation}
Substitute (\ref{eq:sp}), (\ref{eq:htp}) and (\ref{eq:ep}) into (\ref{eq:00}) and (\ref{eq:0A}), and then look at (\ref{eq:sk}) and (\ref{eq:ek}) to confirm that in fact $\BB^\prime = \BB$ and $\OO_A^\prime = \OO_A$ as desired.

For (ii), one checks that the Reeb field $\Th$ of $\that$ satisfies
\begin{equation}
\label{eq:Th}
\Th = e^{-2\Upsilon}(T - 2i\Upsilon^\alpha W_\alpha + 2i\Upsilon^{\bbar}W_{\bbar}),
\end{equation}
while the special defining function $\phat$ associated to $\that$ satisfies 
\begin{equation}
\label{eq:phat}
\phat = e^{2\Upsilon}\phi + O(\phi^2).
\end{equation}
Substituting (\ref{eq:Th}) and (\ref{eq:phat}) into the formulae for the obstruction tensors, and using the orders of vanishing of $\Ein$ given by Theorem~\ref{thm:ache}, yields the result. 

Finally we turn to (iii). First observe that since $\OO_A$ is independent of the choice of approximately Einstein ACH metric, it suffices to choose one good such metric so that $\OO_A$ manifestly satisfies the desired property (\ref{eq:OOA}). 

Fix a contact form and let $g$ be any approximately Einstein ACH metric, given in normal form with tensors $s$ and $\hhat_{AB}$, and with Einstein tensor $\Ein$. For specificity we might as well take $s$ and $\hhat_{AB}$ to be the finite Taylor polynomials of degree $n+1$ given by the proof of Theorem~\ref{thm:ache}. Let now $\gp$ be the approximately Einstein ACH metric characterised by having normal form with
\begin{equation}
\label{eq:skap}
\spr = s + \phi^{n+2}\kappa
\end{equation}
and
\begin{equation}
\label{eq:hlam}
\htp_{AB} = \hhat_{AB} + \phi^{n+2}\lambda_{AB},
\end{equation}
for some prescribed tensors $\kappa$ and $\lambda_{AB}$ on $X$ smooth up to $M$. Evidently prescribing $\kappa$ determines $\NNT^{n+2}s|_M$, and by our assumption that $\BB$ vanishes there is no contradiction in (\ref{eq:sk}). 

If we write $\Einp$ for the Einstein tensor of $\gp$, then substituting (\ref{eq:skap}) and (\ref{eq:hlam}) into (\ref{eq:AB}) gives
\begin{equation}
\label{eq:Einp}
\begin{split}
\phi\Einp_{AB} &= \phi\Ein_{AB} - 2n\phi^{n+2}\lambda_{AB} + (n+2)\phi^{n+2}h^{CD}\lambda_{CD}h_{AB}\\
&\quad\quad + \frac{n}{4}\phi^{n+2}\kappa h_{AB} + O(\phi^{n+3}).
\end{split}
\end{equation}
Multiply both sides by $\phi^{-n-2}$, set $\phi = 0$, and take the trace to yield
\begin{equation}
\label{eq:trEinp}
h^{AB}\phi^{-n-1}\Einp_{AB}|_M = \Big(h^{AB}\phi^{-n-1}\Ein_{AB} + 2n(n+1)h^{AB}\lambda_{AB} + \frac{n^2}{2}\kappa\Big)\Big|_M.
\end{equation}

Now fix the trace part of $\lambda_{AB}$ on $M$ by setting
$$
h^{AB}\lambda_{AB}\Big|_M = -\frac{1}{2n(n+1)}\Big(\frac{n^2}{2}\kappa + h^{AB}\phi^{-n-1}\Ein_{AB}\Big)\Big|_M.
$$
Then (\ref{eq:trEinp}) implies that 
$$
h^{AB}\Einp_{AB} = O(\phi^{n+2}). 
$$
But (\ref{eq:htf}) showed that the trace-free part of $\Einp_{AB}$ is already of order $O(\phi^{n+2})$. Thus $\Einp_{AB} = O(\phi^{n+2})$. Equation (\ref{eq:OOAh}) finishes the proof.
\end{proof}

\section{Volume renormalization and proof of the Main Theorem}
\label{sec:vol}

We normalise the volume form $dv$ of $g$ on $X$ by defining it as
$$
dv := \frac{\sqrt{\det g}}{\sqrt{\det h}} \:d\phi\wedge\theta\wedge (d\theta)^n.
$$
Using (\ref{eq:diag}),
$$
dv = (-\phi)^{-n-2}\frac{\sqrt{s\det\hhat}}{2\sqrt{\det h}}\: d\phi\wedge\theta\wedge (d\theta)^n.
$$
From \S\ref{sec:einstein} we know the power series expansion of $s$ and $\hhat$ up to and including the $\phi^{n+1}$ term. Thus for some locally determined functions $v^{(j)}$ on $M$ we have
\begin{eqnarray*}  
dv &= &\phi^{-n-2}(v^{(0)} + v^{(1)}\phi + v^{(2)}\phi^2 + \cdots +v^{(n+1)}\phi^{n+1} + \\ &&\quad\quad\quad\textnormal{\:\:higher order terms in $\phi$})\: d\phi \wedge\theta\wedge (d\theta)^n.\nonumber
\end{eqnarray*} 

Now pick an $\epsilon_0$ with $-1\ll\epsilon_0<\epsilon<0$. Then
$$
\Vol(\{\epsilon_0<\phi < \epsilon\}) = \int_{\epsilon_0}^\epsilon\int_M dv,
$$
and we have the asymptotic expansion 
\begin{equation*} 
\Vol(\{\epsilon_0<\phi < \epsilon\}) = c_0\epsilon^{-n-1} + c_1\epsilon^{-n} +\cdots + c_n\epsilon^{-1} + L\log(-\epsilon) + V + o(1). 
\end{equation*} 
The constant term $V$ is \textit{renormalized volume}. The coefficients  $c_j$ and $L$ are integrals over $M$ of local TWT invariants (complete contractions of curvature, torsion, Tanno tensor and their covariant derivatives) of  $M$, with respect to the volume element $\theta\wedge (d\theta)^n$. 

\begin{prop}
\label{prop:L}
The log term coefficient $L= \int_M v^{(n+1)}\theta\wedge (d\theta)^n$ is independent of the choice of contact form $\theta$.
\end{prop}
\begin{proof}
The argument is familiar from~\cite{graham} and~\cite{seshadri04}, but we repeat it here for reference. Let $\theta$ and $\that = e^{2\Upsilon}\theta$ be two contact forms on $M$, for $\Upsilon$ a function on $M$, with associated (by Lemma~\ref{lem:special}) special defining functions $\phi$ and $\phat$. So $\phat = e^{2f(x,\phi)}\phi$, for a function $f$ on $\Xb$ near $M$. We can inductively solve the equation $\phat = e^{2f(x,\phi)}\phi$  for $\phi$ to give $\phi = \phat b(x,\phat)$, for a uniquely determined  positive function $b$. Set $\ehat(x,\epsilon) := \epsilon b(x,\epsilon)$; it follows that $\{\epsilon_0 < \phat < \epsilon\}$ is equivalent to $\{\epsilon_0 < \phi < \ehat(x,\epsilon)\}$. Then
\begin{equation*} 
\begin{split} 
\label{eq:voldiff} 
&\Vol(\{\epsilon_0 < \phi < \ehat\}) - \Vol(\{\epsilon_0 < \phi < \epsilon\})= \int_{\epsilon}^{\ehat}\int_Mdv\\  
&\quad = \int_{\epsilon} ^{\ehat}\int_M\phi^{-n-2}(v^{(0)}  + v^{(1)}\phi  + \cdots +v^{(n+1)}\phi^{n+1})d\phi\wedge\theta\wedge (d\theta)^n + o(1). 
\end{split} 
\end{equation*} 
In this expression the $\phi^{-1}$ term contributes $\log b(x,\epsilon)$, so there is no $\log(-\epsilon)$ term as $\epsilon\rightarrow 0$. 
\end{proof}

\begin{thm}
\label{thm:L}
The log term coefficient $L$ is independent of the choice of partially integrable almost CR structure $J$.
\end{thm}
\begin{proof}
The argument parallels that in the proof of~\cite[Theorem 1.1]{gh} and was given in the integrable CR setting in~\cite{seshadri05}. We shall consider a line of partially integrable almost CR structures $J_t$, each with corresponding (by Theorem~\ref{thm:ache}) approximately Einstein ACH metric ${}^tg$ and log term coefficient $L_t$. We then show that $(d/dt)|_{t=0}L_t = 0$. The first variation $(d/dt)|_{t=0}$ of various quantities below will be denoted with a $\bullet$ superscript.

Suppress $t$-dependence for now and recall from Theorem~\ref{thm:ache} that 
\begin{eqnarray*}
\Ric_{\infty\infty} &=& -2(n+2)g_{\infty\infty} + O(\phi^n);\\
\Ric_{00} &=& -2(n+2)g_{00} + O(\phi^n);\\
\Ric_{AB} &=& -2(n+2)g_{AB} + O(\phi^{n+1}).
\end{eqnarray*}
Thus if $\Scal$ denotes the scalar curvature of $g$,
$$
\Scal = -4(n+2)(n+1) + O(\phi^{n+2}).
$$
Its first variation is $\Scal^\bullet = O(\phi^{n+2})$, so for a small negative $\epsilon_0$, as $\epsilon\to 0$, 
$$ 
\int_{\epsilon_0<\phi<\epsilon} \Scal^\bullet dv = O(1). 
$$ 

We shall now work with the fixed coframe $\{d\phi, \theta, \thetat^A\}$ associated with ${}^0g$, the approximately Einstein ACH metric corresponding to a fixed $J_0$. The standard formula (see, e.g.,~\cite[Theorem 1.174]{besse}) for the first variation of the scalar curvature says that
$$ 
\Scal^\bullet = g_{JK,}^{\bullet\phantom{,JK}{JK}} - g_{J\phantom{K},K}^ {\bullet\phantom{J}J\phantom{,K}K} - \Ric^{JK}g_{JK}^\bullet, 
$$ 
or after using the Einstein condition, 
$$ 
\Scal^\bullet =  g_{JK,}^{\bullet\phantom{,JK}{JK}} - g_{J\phantom{K},K}^ {\bullet\phantom{J}J\phantom{,K}K} + 2(n+2)g^{JK}g_{JK}^\bullet + O(\phi^{n+2}). 
$$
Here covariant derivatives are with respect to the Levi-Civita connection of $g$ and indices are raised and lowered using $g$. Integrating gives 
$$ 
\int_{\epsilon_0<\phi<\epsilon} (g_{JK,}^{\bullet\phantom{,JK}{JK}} - g_{J\phantom{K},K}^ {\bullet\phantom{J}J\phantom{,K}K})dv + 4(n+2)\int_{\epsilon_0<\phi<\epsilon}  dv^\bullet= O(1), 
$$ 
where we have used the standard formula~\cite[Proposition 1.186]{besse} in the second integral. Applying the divergence theorem to this gives
$$
-4(n+2)\Vol_{g}^\bullet(\{\epsilon_0<\phi<\epsilon\}) = \int_{\{\phi = \epsilon\}} (g_{JK,}^{\bullet\phantom{,JK}J}  - g_{J\phantom{K},K}^{\bullet\phantom{J}J\phantom{,K}})\nu^Kd\sigma + O(1), 
$$ 
where $\nu$ denotes the unit outward normal to $M^\epsilon = \{\phi = \epsilon\}$ and $d\sigma$ denotes the induced volume element. But $\nu = 2\epsilon\NNT$, hence $\nu^\infty = 2\epsilon, \nu^0 = 0$ and $\nu^A = 0$. Thus if we set
$$ 
\GG := g_{J\infty,}^{\bullet\phantom{,J0}{J}}  - g_{J\phantom{J},\infty}^ {\bullet\phantom{J}J}
$$ 
we have that 
\begin{equation} 
\label{eq:Vdot} 
\Vol_{g}^\bullet(\{\epsilon_0<\phi<\epsilon\}) = -\frac{\epsilon}{2(n+2)}\int_{\{\phi = \epsilon\}}\GG d\sigma + O(1). 
\end{equation}
To prove the theorem, it thus suffices to show that there is no $\log(-\epsilon)$ term on the right-hand side of (\ref{eq:Vdot}).

Let us make the $t$-dependence explicit again and write 
\begin{equation}
\label{eq:tg}
{}^tg = {}^0g_{JK}\thetat^J\odot \thetat^K + tf_{JK}\thetat^J\odot \thetat^{K} + O(t^2). 
\end{equation}
Then  
\begin{equation} 
\label{eq:mathcalA} 
\begin{split}
\GG &= g^{JK}(f_{J\infty,\:K} - f_{JK,\:\infty})\\
            &= \phi^2s^{-1}(f_{0\infty,\:\infty} - f_{00,\:\infty}) - \phi\hhatinv^{AB}(f_{A\infty,\:B} - f_{AB,\:\infty}).
\end{split}
\end{equation} 

In order to obtain an expression for the asymptotic expansion in powers of $\phi$ of $f_{JK}$, we notice that we can also write 
\begin{equation} 
\label{eq:tg2} 
{}^tg = \frac{1}{4\phi^2}{}^t\thetat^\infty\odot {}^t\thetat^\infty + \frac{{}^ts}{\phi^2}{}^t\thetat^0\odot {}^t\thetat^0 - \frac{1}{\phi}{}^t\hhat_{AB}{}^t\thetat^A\odot {}^t\thetat^B, 
\end{equation} 
where $\{{}^t\thetat^J\}, {}^ts$ and ${}^t\hhat_{AB}$ now depend on $t$.

Write 
$$ 
{}^t\thetat^J = \thetat^J + tx_{K}^{\phantom{K}J} \thetat^{K} + O(t^2), 
$$ 
for some $O(1)$ tensor $x$,
$$
{}^ts = s + t\st + O(t^2), 
$$ 
for some $O(1)$ function $\st$, and
$$
{}^t\hhat_{AB} = \hhat_{AB} + t\ft_{AB} + O(t^2), 
$$
for some matrix of $O(1)$ functions $\ft_{AB}$. If we substitute the above three formulae into (\ref{eq:tg2}) and compare with (\ref{eq:tg}) we obtain an expression for $f_{JK}$:
\begin{equation}
\label{eq:f}
\begin{split}
&f_{\infty\infty} = \frac{1}{2\phi^2}x_{\infty}^{\phantom{\infty}\infty};\quad
f_{0\infty} = \frac{1}{2\phi^2}x_{0}^{\phantom{0}\infty} + \frac{2s}{\phi^2}x_{\infty}^{\phantom{\infty}0};\quad
f_{00} = \frac{1}{\phi^2}(\ft_{00} + 2sx_{0}^{\phantom{0}0});\\
&f_{A\infty} = \frac{1}{2\phi^2}x_{A}^{\phantom{A}\infty} - \frac{2}{\phi}\hhat_{AB}x_{\infty}^{\phantom{\infty}B};\quad
f_{AB} = -\frac{1}{\phi}(\ft_{AB} + 2\hhat_{AC}x_{B}^{\phantom{B}C}).
\end{split}
\end{equation}

Now return to (\ref{eq:Vdot}). The leading log term to appear in $d\sigma$ is order $O(\epsilon\log(-\epsilon))$. So a term of order $O(\epsilon^{-2})$ in $\GG$ would combine with this to give a $\log(-\epsilon)$ term on the right-hand side of (\ref{eq:Vdot}). On the other hand, if there were a term in $\GG$ of the form $O(\epsilon^n\log(-\epsilon))$, combination with the lowest order $O(\epsilon^{-n-1})$ term in $d\sigma$ would again produce a $\log(-\epsilon)$ term. To complete the proof of the theorem then it suffices to show that the leading term in $\GG$ is $O(\epsilon^{-1})$ and the leading log term in  $\GG$ is $O(\epsilon^{n+1}\log(-\epsilon))$. But this is seen to be true by examining (\ref{eq:mathcalA}) and (\ref{eq:f}) while using the definition of Levi-Civita covariant differentiation
$$
f_{A\infty,\:B} = W_Bf_{A\infty} - f_{J\infty}\psi_{A\phantom{J}B}^{\phantom{A}J} - f_{AJ}\psi_{\infty\phantom{J}B}^{\phantom{\infty}J},
$$
together with Proposition~\ref{prop:levicivita}.
\end{proof} 

Since the partially integrable almost CR structures admitted by $(M,H)$ form a contactible homotopy class, Theorem~\ref{thm:L} combined with Proposition~\ref{prop:L} proves the Main Theorem. Furthermore we have:

\begin{prop}
The log term coefficient $L$ is invariant under deformations of the contact structure $H$.
\end{prop}
\begin{proof}
This is a standard application of the well-known result of Gray~\cite{gray} that there are no nontrivial deformations of contact structures. Precisely, if $\{H_\tau\}$ is a one-parameter family of contact structures on $M$ through a specified contact structure $H = H_0$, then there exist a family of diffeomorphisms $\{\FF_\tau\}$ of $M$ such that $H_\tau = (\FF_\tau)_\ast H$. We fix $\tau$ and shall show that the values of $L$ for the contact structures $H$ and $H_\tau$ are the same. 

By Proposition~\ref{prop:L} and Theorem~\ref{thm:L} we may fix a choice of some $\theta$ and $J$ for $H$. Then for $H_\tau$ use $\theta_\tau := (\FF_\tau^{-1})^*\theta$ and $J_\tau := (\FF_\tau^{-1})^*J(\FF_\tau)_\ast$. Let $\phi$ and $\phi_\tau$ be the special defining functions associated respectively to $(\theta, J)$ and $(\theta_\tau, J_\tau)$. Then the argument follows that in the proof of Proposition~\ref{prop:L}, replacing $\phat$ there with $\phi_\tau$.
\end{proof}

\section{The integrable case}
\label{sec:integr}

Suppose that the partially integrable almost CR structure $J$ on $M$ is in fact integrable, i.e., satisfies $[H^{1,0}, H^{1,0}]\subset H^{1,0}$. Then $(M, H, J)$ is a \emph{CR manifold}. If $n \geq 2$, so that $M$ has dimension five or greater, it is then a classical result that $M$ forms the boundary of a strictly pseudoconvex domain in a complex manifold. Living on a one-sided neighbourhood $X$ of $M$ is the approximately Einstein complete K\"ahler metric $g_+$ of Fefferman~\cite{fef76}. Volume renormalization with respect to $g_+$ was carried out in~\cite{seshadri04}. 

If $n=1$, $M$ is three-dimensional and $J$ is automatically integrable. On the other hand, in this dimension, $M$ does not in general admit a global CR embedding. In the special case where three-dimensional $M$ \emph{is} embeddable, Fefferman's approximately Einstein complete K\"ahler metric may be used for volume renormalization in the same way as for higher dimensions (\cite{seshadri04}). If three-dimensional $M$ is \emph{not} embeddable the methods of~\cite{seshadri04} do not work. This general case is instead handled by Herzlich~\cite{herzlich}. The relation between our volume renormalization methods in this paper and those of Herzlich are not clear. However, let us remark that when $n=1$ the log term coefficient $L$ always vanishes for both approaches by basic invariant theory, see, e.g., the proof of Theorem 9.1 in~\cite{bhr}. Since the present paper is predominantly concerned with this log term coefficient, for this section we shall assume $n\geq 2$ and investigate the relation between this paper and our earlier work~\cite{seshadri04}.

Let us briefly recall some basic facts from~\cite{seshadri04}. For an  arbitrary defining  function $\phi$ for $M = M^0$, each  level set $M^{\epsilon} := \{\phi = \epsilon\}$ is strictly  pseudoconvex with the natural CR bundle $\HHH^{1,0}_\epsilon := T^{1,0}\Xb\cap\C TM^\epsilon$ and contact form $(i/2)(\overline{\partial}\phi - \partial\phi)|_{M^\epsilon}$. Let  $\HHH^{1,0}\subset\C T\Xb$ denote the bundle whose fibre over each $M^{\epsilon}$ is $\HHH^{1,0}_{\epsilon}$.  The restriction of  $i\partial\overline{\partial}\phi$ to $\HHH^{1,0}$ is positive definite. This means $\partial\overline{\partial}\phi$ has  precisely one null direction transverse to $\HHH^{1,0}$, whence there is a  uniquely defined $(1,0)$ vector field $\xi$ that satisfies 
$$ 
\xi\perp_{\partial\overline{\partial}\phi}\HHH^{1,0};\quad\partial\phi(\xi) = 1. 
$$  

Let  $\{W_{\alpha}\}$ be any local frame for $\HHH^{1,0}$. Since  $\xi$ is transverse to $\HHH^{1,0}$, the set of vector fields $\{W_{\alpha}, \xi\}$  is a local frame for $T^{1,0}\Xb$. The dual $(1,0)$ coframe is then  of the form $\{\vartheta^\alpha, \partial\phi\}$ for some $(1,0)$-forms  $\{\vartheta^{\alpha}\}$ that annihilate $\xi$. We may  write 
$$
\ddbar\phi = h_{\al\bbar}\vartheta^{\alpha}\wedge\vartheta^{\bbar}  + r\partial\phi \wedge\dbar\phi,
$$ 
for a positive definite Hermitian matrix of functions $h_{\al\bbar}$ and a real-valued function $r$. 

We may identify $M \times (-1,0]$ with $\Xb$ by following the unit-speed integral curve of $\Re\:\xi$, emanating from a point $x$ on $M$, for time $\phi$. 
 
Now the K\"ahler form $\omega$ of $g_+$ is given as
$$
\omega = \ddbar\log(-1/\rho),
$$
where $\rho$ is an approximate Monge--Amp\'ere defining function (see~\cite{fef76},~\cite{seshadri04}). Using this $\rho$ in place of $\phi$ in the previous discussion, we may write
$$
g_+ = -\frac{1}{\rho}h_{\al\bbar}\vartheta^{\alpha}\odot \vartheta^{\bbar} +  \frac{1-r\rho}{\rho^2}\partial\rho\odot\dbar\rho.
$$
One can now verify that $g_+$ is ACH according to Definition~\ref{df:ach}. In fact it was consideration of this special K\"ahler case that led us to Definition~\ref{df:ach} in the first place.

Moreover the special defining function $\phi$, for a contact form $\theta$, used in~\cite{seshadri04} is characterised by the existence of $(1,0)$-forms $\{\varthetat^{\alpha}\}$, with $\varthetat^{\alpha}|_M = \vartheta^\alpha|_M$, such that 
\begin{equation}
\label{eq:g_+}
g_+ = -\frac{1}{\phi}\hhat_{\al\bbar}\varthetat^{\alpha}\odot \varthetat^{\bbar} +  \frac{1}{\phi^2}\partial\phi\odot\dbar\phi,
\end{equation}
and 
$$
\frac{i}{2}(\dbar\phi - \partial\phi)|_M = \theta,
$$
where $\hhat_{\al\bbar}$ is a positive definite Hermitian matrix of functions with $\hhat_{\al\bbar}|_M = h_{\al\bbar}$. Setting $\vartheta := (i/2)(\dbar\phi - \partial\phi)$, (\ref{eq:g_+}) can be rewritten as
$$
g_+ = -\frac{1}{\phi}\hhat_{\al\bbar}\varthetat^{\alpha}\odot \varthetat^{\bbar} +  \frac{1}{4\phi^2}d\phi\odot d\phi + \frac{1}{\phi^2}\vartheta\odot\vartheta.
$$
Replacing $\vartheta$ with $(1/2)\vartheta$ gives instead
$$
g_+ = -\frac{2}{\phi}\hhat_{\al\bbar}\varthetat^{\alpha}\odot \varthetat^{\bbar} +  \frac{1}{4\phi^2}d\phi\odot d\phi + \frac{4}{\phi^2}\vartheta\odot\vartheta.
$$
Comparing this expression with (\ref{eq:diag}), (\ref{eq:hhatM}) and Lemma~\ref{lem:special} it is clear that the special defining function of~\cite{seshadri04} coincides with that used in the present paper.

Now from~\cite{fef76} and~\cite{seshadri04} (taking care with conventions for Ricci and K\"ahler forms), $g_+$ satisfies the approximately Einstein condition stated in Theorem~\ref{thm:ache}. Thus by the second assertion of that theorem, $g_+$ coincides, as a Riemannian metric, with an approximately Einstein ACH metric, modulo high-order error terms and diffeomorphism action. Therefore the asymptotic expansion for the volume form of $g_+$ produced via the complete K\"ahler approach in~\cite{seshadri04} must coincide with that produced via the ACH methods in this paper. In particular, the log term coefficients $L$ produced by by the two approaches must agree. 

Turning to the obstruction tensors, it is well known that in the integrable case the obstruction to $g_+$ being an Einstein--K\"ahler metric (i.e., having Einstein tensor vanishing to all orders), with $\phi g_+|_H$ extending smoothly to $M$, is given purely as a scalar function $f$. The function $f$ may be identified with the boundary value of the first log term in the asymptotic expansion, in powers of Fefferman's approximate solution $\rho$, for the solution to a complex Monge--Amp\'ere equation. It is known to transform according to $\fh = e^{-2(n+2)\Upsilon}f$ under a change in contact form $\that = e^{2\Upsilon}\theta$.  See~\cite{fef76},~\cite{lm} and~\cite{graham87} for details. It remains to be investigated the precise relation between $f$ and the obstructions $\BB$ and $\OO_A$ appearing in Theorem~\ref{thm:ache}. 

\section{Concluding remarks}
\label{sec:remarks}

\subsection{CR $Q$-curvature} As alluded to in the Introduction, one of the motivating factors behind this work was to shed some light on the mysterious CR $Q$-curvature. This quantity was introduced by Fefferman--Hirachi~\cite{fh}, as an analogy of a quantity in conformal geometry. In~\cite{seshadri04} we gave an alternative description of $Q^{\textnormal{CR}}$ and further showed that
$$
\int_M Q^{\textnormal{CR}}_\theta\: \theta\wedge (d\theta)^n = \const \times L,
$$
where $L$ is the log term coefficient in the volume renormalization of Fefferman's approximately Einstein complete K\"ahler metric.

The present work has shown that $L$ may be defined for only partially integrable almost CR structures and that it is moreover a contact invariant. It is natural to try to generalise CR $Q$-curvature to this partially integrable setting. In fact this may be done fairly easily, proceeding analogously to~\cite[Appendix A]{seshadri04}. Tentatively calling this quantity ``contact $Q$-curvature'', it is also not hard to show that
$$
\int_M Q^{\textnormal{contact}}_\theta\: \theta\wedge (d\theta)^n = \const \times L,
$$
where this time $L$ is the log term coefficient in the volume renormalization of our approximately Einstein ACH metric.

The difficulty arises in comparing $Q^{\textnormal{contact}}$ defined by this method with a definition analogous to that of $Q^{\textnormal{CR}}$ in~\cite{fh}. Since Fefferman--Hirachi's definition of $Q^{\textnormal{CR}}$ uses Fefferman's conformal structure associated to a CR manifold, one would need to use a generalisation of that structure to the partially integrable case. Such a generalisation does exist, see~\cite{bd}, however while in the integrable case Fefferman's conformal structure and ambient metric are both intimately related to Fefferman's approximately Einstein complete K\"ahler metric, the relation between the generalised Fefferman structure and a suitable ambient metric, and the approximately Einstein ACH metric is not clear.

In dimension three, we know that $L$, or equivalently $\int_M Q^{\textnormal{CR}}$, always vanishes. Hirachi and others have asked the question ``Does the integral of CR $Q$-curvature always vanish in higher dimensions?''. One would have an affirmative answer to this question if it were that $L$ for partially integrable almost CR structures always vanishes. As we speculate in the next subsection, the contact-invariance of $L$ could help to settle the question of its vanishing.

\subsection{Contact invariants} In recent years, several invariants of contact structures have been defined via local differential geometric techniques. Boutet de Monvel~\cite{boutet04} proved that the logarithmic trace of generalised Sz\"ego projectors is a contact invariant, but later in~\cite{boutet06} showed that it always vanishes. Ponge~\cite{ponge} develops a large class of contact invariants via the noncommutative residue traces of Heisenberg-pseudodifferential projections. Biquard--Herzlich--Rumin~\cite{bhr} showed that the residue at zero of an eta function coming from Rumin's contact complex is a contact invariant. In~\cite{seshadri07} we prove that the regular value at zero of a well-chosen combination of zeta functions coming from the contact complex is a contact invariant. However, no nonvanishing examples of these invariants are known. (Incidentally, vanishing of Biquard--Herzlich--Rumin's invariant is a necessary and sufficient condition for defining an eta invariant of the contact complex---see~\cite[\S 9]{bhr}.) The interested reader should read Ponge's discussion in~\cite[\S 4.3]{ponge} for further information.

It is likely that the invariants just described are of the same form as the contact invariant $L$ described in this paper, namely, they are the integrals of local TWT invariants. This reminds us of a result of Gilkey~\cite{gilkey}, which settles a famous conjecture of I.~M.~Singer. Roughly stated it says the following:
\begin{thm}[Gilkey~\cite{gilkey}]
Suppose there is a scalar Riemannian invariant whose integral is independent of the choice of Riemannian metric. Then the smooth topological invariant given by this integral is (a known constant multiple of) the Euler characteristic.
\end{thm}

In particular there are \emph{no} topological invariants of odd dimensional manifolds given by integrating local Riemannian invariants. We wonder if an analogous result holds for contact manifolds, so we close with the following:
\begin{question}
If a contact invariant is given as the integral of local TWT invariants, does it always vanish?
\end{question}

\bibliographystyle{smfplain} 
\bibliography{biblio} 

\end{document}